\documentclass[11pt]{article}

\usepackage{ulem}
\usepackage{amssymb,amsbsy,amsmath,amsfonts,amssymb,amscd,amsthm}
\usepackage{xcolor}
\usepackage{graphicx} 
\usepackage{multicol}
\graphicspath{{./}{figures/}} 
\usepackage{tikz}
\usetikzlibrary{matrix}

\newcommand{\norm}[1]{\|#1 \|}

\newcommand{\R}{\mathbb{R}}

\newcommand{\N}{\mathbb{N}}
\newcommand{\T}{{\mathbb{T}^d}}

\newcommand{\mP}{\mathcal P}

\renewcommand{\H}{\mathcal{H}}

\newcommand{\beq}{\begin{equation}}
\newcommand{\eeq}{\end{equation}}
\def\a{\alpha}

\def\g{\gamma}
\def\l{\lambda}
\def\L{\Lambda}

\def\s{\sigma}

\def\th{\theta}

\newcommand{\cF}{{\cal F}}

\newcommand{\cH}{{\cal H}}

\newcommand{\cP}{{\cal P}}

\newcommand{\diver}{{\rm div}}
\def\pd{\partial}
\def\half{\frac{1}{2}}
\newtheorem{theorem}{Theorem}[section]
\newtheorem{lemma}[theorem]{Lemma}

\newtheorem{corollary}[theorem]{Corollary}
\newtheorem{remark}[theorem]{Remark}

\numberwithin{equation}{section}

\title{Rates of convergence for the policy iteration method for   Mean Field Games systems}

\author{Fabio Camilli \and Qing Tang}
\date{\today} 
\begin{document}
\maketitle
\begin{abstract}
Convergence of the policy iteration method for discrete and continuous optimal control problems 
holds under general assumptions.  Moreover, in some circumstances, it is also possible to show  a quadratic rate of convergence for the algorithm. For Mean Field Games, convergence of the policy iteration method has been recently proved in \cite{ccg}. Here, we provide an estimate of its rate of convergence.
\end{abstract}
\noindent
{\footnotesize \textbf{AMS-Subject Classification:} 49N80; 35Q89; 91A16; 65N12}.\\
{\footnotesize \textbf{Keywords:} Mean Field Games, policy iteration, convergence rate}.

\section{Introduction}
The policy iteration method, introduced by Bellman \cite{b} and Howard \cite{h}, is a general procedure to solve the  Hamilton-Jacobi-Bellman (HJB in short) equation, a nonlinear equation arising in discrete and continuous optimal control problems. The solution of  the  nonlinear HJB equation, which it is  well known to suffer from   the so-called ``curse of dimensionality" (see \cite{b}), is replaced  by the solution of a sequence of linear problems, coupled at each step with an optimization problem for  the updating of the new policy.
It has been proved that, under general assumptions, the algorithm converges  to the solution of the original problem (see \cite{alla,fl,pu1,pb,santos}); moreover, in some cases, it is possible to show  a (local) quadratic rate of convergence of the method which explains its very rapid convergence   to the solution of the original problem  (see \cite{bmz,kss}).\par
Mean Field Games (MFG in short) theory has been introduced in \cite{hcm,ll} to characterize Nash equilibria for differential games involving a large (infinite) number of agents. The corresponding mathematical formulation
leads to the study of a system of PDEs, composed by a HJB equation, characterizing the value function
and the optimal control for the agents; a Fokker-Planck (FP in short) equation, governing the distribution of the   population when the agents behaves in an optimal way. In the case of a finite horizon problem with periodic boundary conditions, the MFG system reads as 
\begin{equation}\label{MFG}
\begin{cases}
-\partial_tu-\Delta u+H(Du)=\sigma F[m(t)](x) & \text{ in }Q\\
\partial_tm -\Delta  m-\diver(mH_p(Du))=0 & \text{ in }Q\\
m(x,0)=m_0(x),\, u(x,T)=u_T(x) & \text{ in }\T\ ,
\end{cases}
\end{equation}
where $Q:=\T\times[0,T]$, $\T$ stands for the flat torus $\R^d / \mathbb{Z}^d$,   $H$ is a convex Hamiltonian and $\sigma$ is a positive constant.
The difficulty in solving the previous system can prevent the concrete application of the MFG model to real life problems. Indeed, system \eqref{MFG} not only involves the resolution of a HJB equation, but  the two equations are strongly coupled  and the system   has a forward-backward structure which does not permit to solve the two equations in parallel. For this reason, strategies other than the simple discretization of \eqref{MFG} must be implemented
(see for example \cite{bks,ch}).\par
In \cite{ccg},  the following version   of the policy iteration method for the MFG system \eqref{MFG} was considered:\par

\textbf{Policy iteration algorithm:} Fixed $R>0$ and given a bounded, measurable vector field   $q^{(0)}:\T\times [0,T]\to\R^d$  with $\|q^{(0)}\|_{L^\infty(Q)}\le R$,  iterate
\begin{itemize}
	\item[\textbf{(i)}] Solve
	\begin{equation*}
		\left\{
		\begin{array}{ll}
			\pd_t m^{(n)}-\Delta m^{(n)}-\diver (m^{(n)} q^{(n)})=0,\quad &\text{ in }Q\\
			m^{(n)}(x,0)=m_0(x)&\text{ in }\T.
		\end{array}
		\right.
	\end{equation*}
	\item[\textbf{(ii)}] Solve
	\begin{equation*}
		\left\{
		\begin{array}{ll}
			-\pd_t u^{(n)}- \Delta u^{(n)}+q^{(n)} Du^{(n)}-L(q^{(n)})=\sigma F[m^{(n)}(t)](x)&\text{ in }Q\\
			u^{(n)}(x,T)=u_T(x)&\text{ in }\T.
		\end{array}
		\right.
	\end{equation*}
	\item[\textbf{(iii)}] Update the policy
	\begin{equation*}
		q^{(n+1)}(x,t)={\arg\max}_{|q|\le R}\left\{q\cdot Du^{(n)}(x,t)-L(q)\right\}\quad\text{ in }Q.
	\end{equation*}
\end{itemize}
Here     $L(q)=\sup_{p\in\R^d}\{p\cdot q-H(p)\}$ is the   Legendre transform of $H$. 
Compared to the algorithm for the HJB equation alone, each iteration of the previous method also includes   the resolution of the FP equation. The main advantage of the method, in addition to what has already been observed previously for the HJB equation, is that at each iteration the linear HJB and FP equations are completely decoupled and can be quickly solved with different numerical methods. In \cite{ccg}, it was    proved    that   the previous algorithm converges to a solution $(u,m)$ of the MFG system if the Hamiltonian $H$ is Lipschitz continuous or if $H (p)\simeq |p|^\g$, $\g>1$.\par

In this paper, we study   rates of convergence for the MFG policy iteration method. We obtain, via purely PDE methods, a linear rate of convergence for the solution to the HJB equation in MFG system. More precisely, the error between two successive iterations of the sequence  $\|u^{(n)}\|_{B_1}+\s \|m^{(n)}\|_{B_2}$ generated by the algorithm improves linearly with respect to the error of the previous iterations, for sufficiently large $n$ and small $\s$, where $\|\cdot \|_{B_i}$, $i=1,2$,  denotes some Banach space norm which will be specified later. Without the coupling, for the HJB equation, the policy iteration method  is  equivalent to the Newton's method (see \cite{pb} and also Section \ref{sec:Newton_MFG}) and therefore gives a (local) quadratic improvement with respect to the error $|u^{(n)}-u|$ at each step. However, since the policy $q^{(n)}=H_p (Du^{(n-1)})$ enters as a velocity field in the FP equation, its improvement can correspond at most to a linear one   for the distribution error $|m^{(n)}-m|$. This also reflects on the HJB equation due to the coupling term $F[m]$ on the right side of this equation.  This point is further explained by   an interpretation of the policy iteration method  for the MFG system   as a quasi-Newton's method obtained, in order the eliminate the coupling among the equations, by suppressing off-diagonal elements in Jacobian of the map   of which we are calculating the roots.\par
Despite the previous limitations, however, the policy iteration method retains the advantage of replacing the resolution of a strongly coupled nonlinear system with a sequence of decoupled linear problems. Moreover, in a neighbourhood of the solution, the rapid convergence of the value function  $u^{(n)} $ is also reflected in an equally rapid convergence of the distribution $m^{(n)} $  (see \cite{ccg}  and \cite{lst} for some numerical simulations). 
\par
In this paper we restrict the discussion to MFGs with separable Hamiltonians. Recently in \cite{lst} the authors considered the convergence rate of policy iteration algorithms for MFGs with non-separable Hamiltonians using contraction fixed point method. The key difference is that here we do not impose the short time assumption, which is essential for the reasoning in \cite{lst}. \par

The paper is organized as follows. In Section \ref{sec:prelim}, we introduce some notations and recall the convergence result in \cite{ccg}. In Section \ref{sec:estimate} and \ref{sec:stat}, we prove the convergence rate for the  MFG policy iteration method for the evolutive problem and, respectively, for the stationary one. In Section \ref{sec:Newton_MFG}, we provide the  interpretation of the policy iteration algorithm as a quasi-Newton method.


\section{The policy iteration method for the Mean Field Games system}\label{sec:prelim}
In this section, we recall some results about the convergence of the policy iteration method for the MFG system.\par
Throughout the paper we work with maps which are periodic in space, i.e. on the torus $\T$. This simplification allows us to ignore problems related to boundary conditions or growth conditions on the data. The main ideas of this paper can be extended to consider models in $\R^d$ . We denote  by $L^r(\T)$, $1\le r\le\infty$, the set of $r$ summable periodic functions and by $W^{k,r}(\T)$, $k\in\N$ and $1\le r \le\infty$, the Sobolev space of periodic functions having $r$-summable weak derivatives up to order $k$. For any $r\geq1$, we denote by $W^{2,1}_r(Q)$ the space of functions $f$ such that $\partial_t^{\delta}D^{\beta}_x f\in L^r(Q)$ for all multi-indices $\beta$ and $\delta$ such that $|\beta|+2\delta\leq  2$. All these spaces  are endowed with the corresponding standard norm. \\
Defined $W^{1,0}_s(Q)$ as the space of functions such that the   norm
\[
\norm{u}_{W^{1,0}_s(Q)}:=\norm{u}_{L^s(Q)}+\sum_{|\beta|=1}\norm{D_x^{\beta}u}_{L^s(Q)}
\]
is finite, we denote by $\cH_s^{1}(Q)$ the space of   functions $u\in W^{1,0}_s(Q)$ with $\partial_t u\in (W^{1,0}_{s'}(Q))'$, where $\frac 1 s+\frac 1 {s'}=1$, equipped with the natural norm
\begin{equation*}
	\norm{u}_{\mathcal{H}_s^{1}(Q)}:=\norm{u}_{W^{1,0}_s(Q)}
	+\norm{\partial_tu}_{(W^{1,0}_{s'}(Q))'}\ .
\end{equation*}
If   $s>d+2$, then   $\cH_s^1(Q)$ is continuously embedded onto $C^{\delta,\delta/2}(Q)$ for some $\delta\in (0,1)$, see \cite[Appendix A]{metafune_altri}.\\
We   describe the assumptions on the data of the problem. 
Concerning the Hamiltonian, we  consider  two different settings
\begin{itemize}
	\item[\textbf{(H1)}] $H$ is differentiable, convex and globally Lipschitz  continuous, i.e.  there exists  a constant $R_0>0$ such that
	\begin{equation*} 
		|D_pH(p)|\leq R_0\qquad\text{ for all }p\in\R^d\ .
	\end{equation*}
	\item[\textbf{(H2)}] $H$ is of the form
	\begin{equation*} 
		H(p)=|p|^\gamma, \qquad \gamma>1.
	\end{equation*}
\end{itemize}
Recall that
\begin{equation*}
	H(p) = p\cdot q - L(q) \quad \text{if and only if} \quad q = D_p H(p)\ .
\end{equation*}
Concerning the coupling cost, we assume that 
\begin{itemize}
	\item[\textbf{(F1)}] $F$ maps continuously $\cP(\T)$, the set of probability measure on $\T$, into a bounded subset of $C^{0,1}(\T)$. Moreover
	\[
		\int_{\T}(F[m_1] -F[m_2] )d(m_1-m_2)>0 \qquad \text{if $m_1\neq m_2$.}
	\]
	for $x_1,x_2\in\T$, $m_1,m_2\in\mP_1(\T)$.
\end{itemize}
Finally,  for the initial and terminal data, we suppose   that 
\begin{itemize}
	\item[\textbf{(I)}] 
	$u_T\in W^{2-\frac{2}{r},r}(\T)$, $r>d+2$;\\
	$m_0\in W^{1,s}(\T)$, $s>d+2$,  $m_0\ge 0$ and $\int_{\T}m_0(x)dx=1$.
\end{itemize}
	
In the following we recall some a priori estimate for the solution of the linear  equations involved in policy iteration method
(see \cite[Lemma 2.1 and 2.2]{ccg})
\begin{lemma}\label{lemma:linear_HJ}
	Given  $b\in L^\infty(Q;\R^d)$, $f\in L^r(Q)$ and $u_T\in W^{2-\frac{2}{r},r}(\T)$ for some $r>1$, then the problem
	\begin{align*}
		\begin{cases}
			-\partial_t u-\Delta u+b(x,t)   Du=f(x,t)&\text{ in }Q\\
			u(x,T)=u_T(x)&\text{ in }\T
		\end{cases}
	\end{align*}	
	admits a unique solution $u\in W^{2,1}_r(Q)$  and it holds
	\begin{equation*}
		\|u\|_{W^{2,1}_r(Q)}\leq C(\|f\|_{L^r(Q)}+\|u_T\|_{W^{2-\frac{2}{r},r}(\T)}),
	\end{equation*}
	where $C$ depends on the norm of the coefficients as well as on $r,d,T$. Furthermore, if $r>d+2$ we have $Du\in C^{\alpha,\alpha/2}$ for some $\alpha\in(0,1)$.	
\end{lemma}
\begin{lemma}\label{lemma:linear_FP}
Given a bounded, measurable vector field $g:Q\to\R^d$  and  $m_0\in L^2(\T)$, $m_0\geq0$, then the problem
\[
	\left\{
	\begin{array}{ll}
		\pd_t m-  \Delta m-\diver (g(x,t)m)=0&\text{ in }Q,\\
		m(x,0)=m_0(x)&\text{ in }\T,
	\end{array}
	\right.
\]
has a unique  non negative  solution $m\in \H_2^1(Q)$. Furthermore, if $m_0\in L^s(\T)$, $s\in(1,\infty)$, then $m\in L^\infty(0,T;L^s(\T))\cap \H_2^1(Q)$ and, if $m_0\in W^{1,s}(\T)$, then
\begin{equation*}
	\|m\|_{\H_s^1(Q)}\le C
\end{equation*}
for some constant  $ C=C(\|g\|_{L^\infty(Q;\R^d)},\|m_0\|_{W^{1,s}(\T)})$.
\end{lemma}

In \cite[Theorems 2.3 and 2.5]{ccg}, it has been proved the following convergence result for the policy iteration method.
\begin{theorem}\label{thm:policy_iteration}
		Let either   \textbf{(H1)} or  \textbf{(H2)},   \textbf{(F1)} and \textbf{(I)} be  in force. Then, for $R$ sufficiently large, the sequence $(u^{(n)},m^{(n)})$, generated by the policy iteration algorithm, converges    to the solution    $(u ,m)\in W^{2,1}_r(Q)\times   \H_s^{1}(Q)$   of \eqref{MFG}. 	
	\end{theorem}
\begin{remark}
If \textbf{(H1)} holds, one can write  $H(p)=\sup_{|q|\le R_0}\{p\cdot q-L(q)\}$. Therefore, in this case, it is sufficient to consider $R=R_0$ in the policy iteration algorithm to get a converging sequence to the solution of \eqref{MFG}. \\
If  \textbf{(H2)} holds  and $(u,m)$ is the solution of \eqref{MFG}, then there exists a constant $\overline R$ such that $\|Du(t)\|_{L^\infty(\T)}\le {\overline R}$ for any $t\in [0,T]$
(see \cite[Lemma 2.4]{ccg}). Then  one introduces the truncated Hamiltonian defined by
\begin{equation*}
	H_{\overline R}(p)=\begin{cases}
		|p|^\gamma&\text{ if }|p|< {\overline R}\ ,\\
		(1-\gamma){\overline R}^\gamma+\gamma {\overline R}^{\gamma-1}|p|&\text{ if }|p|\geq {\overline R}\ ,
	\end{cases}	
\end{equation*}
and the problem
\begin{equation}\label{MFGS}
	\begin{cases}
		-\partial_tu -\Delta u +H_{\overline R}(Du)=\s F[m(t)](x) & \text{ in }Q,\\
		\partial_tm  -\Delta  m -\diver(m D_pH_{\overline R}(Du))=0 & \text{ in }Q,\\
		m(x,0)=m_0(x),\, u(x,T)=u_T(x) & \text{ in }\T\ .
	\end{cases}
\end{equation}
Observe that a solution $(u,m)$ of \eqref{MFG} is also a solution of \eqref{MFGS} and  $H_{\overline R}$ is globally Lipschitz continuous. Since $H_{\overline R}$ satisfies assumption \textbf{(H1)}, the policy iteration method  with $R={\overline R}$  converges to the solution of \eqref{MFGS} and therefore also of \eqref{MFG}.\\
Note also that, by the  Sobolev  embedding of $W^{2,1}_r(Q)$ in $C^{1+\a,\frac{1+\a}{2}}(Q)$ for $r>d+2$ with
$$\|u\|_{C^{1+\a,(1+\a)/{2}}(Q)} \le C\|u\|_{W^{2,1}_r(Q)}$$
(see  \cite[Cor. IV.9.1]{LSU})
and since $q^{(n)}=H_p(Du^{(n-1)})$, it also follows the convergence of policy $q^{(n)}$ to the optimal control $q=H_p(Du)$
in $L^\infty(Q)$ for $n\to\infty$.
\end{remark}
\begin{remark}
	Sobolev regularities for solutions to the Fokker-Planck equation in the MFG system have been also considered in \cite[Section 10.3]{ben}.
\end{remark}
\section{A rate of convergence for the policy iteration method: the finite horizon problem}\label{sec:estimate}
In this section, we give an estimate of the rate of convergence for the policy iteration method.  We  replace 
assumption \textbf{(H1)} with
\begin{itemize}
	\item[\textbf{(H3)}]  $H$  is two times differentiable, satisfies \textbf{(H1)} and for any $S>0$, there  exists $C_S>0$ such that
	\begin{equation*} 
		H_{pp}(p)q\cdot q\le C_S|q|^2 \quad \text{for any    $|p|\le S$, $q\in\R^d$}.
	\end{equation*}
\end{itemize}
and \textbf{(F1)} with

\begin{itemize}
	\item[\textbf{(F2)}] 
	$F:\T\times L^s(\T)\to L^r(\T)$	 and for all $t\in (0,T)$,
	\begin{equation*} 
		\|F[m_1]-F[m_2]\|_{L^r(\T)}\le C_F  \|m_1 -m_2\|_{L^s(\T)},
	\end{equation*}
	for $r,s>d+2$ and all  $m_1,\,m_2\in \cH^1_s(Q)$.	
\end{itemize}


We prove an estimate for the rate of convergence for the policy iteration method. 
\begin{theorem}\label{prop:est1}
	Let either  \textbf{(H2)} or \textbf{(H3)},   \textbf{(F2)} and  \textbf{(I)} be  in force and $R$  as in Theorem \ref{thm:policy_iteration}. Then, there exists a constant $C$, depending only on the data of problem, such that, if $(u^{(n)}, m^{(n)})$ is the sequence generated by the policy iteration method, we have
	\begin{equation}\label{mCs}
		\begin{split}
			\|m^{(n+1)} - m^{(n)}\|_{C(0,T;L^s(\T))}&\le C \|q^{(n+1)}-q^{(n)}\|_{ L^\infty (Q)},
		\end{split}
	\end{equation}
	 \begin{equation}\label{mH12}
		\begin{split}
			\|m^{(n+1)} - m^{(n)}\|_{\H_2^1(Q)}&\le C \|q^{(n+1)}-q^{(n)}\|_{ L^\infty (Q)},
		\end{split}
	\end{equation}	 
	\begin{equation} \label{uW21r}
		\begin{split}
	 \|u^{(n+1)}-u^{(n)}\|_{W^{2,1}_r(Q)}
	 &\le C\big(\|u^{(n)}-u^{(n-1)}\|^2_{W^{2,1}_r(Q)}\\
	 &+ \sigma \|m^{(n+1)} - m^{(n)}\|_{C(0,T;L^s(\T))}\big).
	 \end{split}
	\end{equation}
\end{theorem}


\begin{proof}
Along the proof, the constant $C$ can change from line to line, but it is always independent of $n$.\\
Set $M^{n+1}=m^{(n+1)}-m^{(n)}$. Then $M^{n+1}$ satisfies the equation
\begin{equation}\label{eq1}
	\pd_t M^{n+1}- \Delta M^{n+1}-\diver (q^{(n+1)} M^{n+1})=\diver ((q^{(n+1)}-q^{(n)}) m^n).
\end{equation}
with $M^{n+1}(0)=0$. \\
We first show \eqref{mCs}. The proof follows the techniques from \cite[Lemma 4.1]{c1999}.
We set  
$$
f_0=q^{(n+1)}M^{n+1}+(q^{(n+1)}-q^{(n)})m^{(n)},
$$
and, for any $\tau$ such that $0<\tau<t$, we notice
\begin{align*}
	&\frac{d}{d\tau}|M^{n+1}(   \tau)|^s\\
	=&s|M^{n+1}( \tau)|^{s-2}M^{n+1}(  \tau) \frac{\partial M^{n+1}(  \tau)}{\partial \tau}\\
	=&s|M^{n+1}(  \tau)|^{s-2}M^{n+1}(  \tau) \big(\Delta M^{n+1}( x,\tau)+\diver f_0(  \tau)\big).
\end{align*}
Integrating the previous relation in $\T$ and observing that 
\begin{align*}
	&\int_{\T}s|M^{n+1}( \tau)|^{s-2}M^{n+1}(   \tau) \Delta M^{n+1}(    \tau)dx\\
	&=-s(s-1)\int_{\T}|M^{n+1}(   \tau)|^{s-2}|DM^{n+1}(   \tau)|^2dx\\
	&=-s(s-1)\int_{\T}\big(|M^{n+1}(  \tau)|^{(s-2)/s}|DM^{n+1}(   \tau)|^{2/s}\big)^sdx,
\end{align*}
we get
\begin{equation}\label{id1}
	\begin{split}
		&\int_{\T}\frac{d}{d\tau}|M^{n+1}(   \tau)|^sdx\\
		&+s(s-1)\int_{\T}\big(|M^{n+1}(    \tau)|^{(s-2)/s}|DM^{n+1}(    \tau)|^{2/s}\big)^sdx\\
		&= s\int_{\T}|M^{n+1}(    \tau)|^{s-2}M^{n+1}(    \tau) \big(\diver f_0(    \tau)\big)dx\\
		&=-s(s-1)\int_{\T} |M^{n+1}(    \tau)|^{s-2}DM^{n+1}(    \tau)f_0(    \tau)dx\\
		&\leq s(s-1)\int_{\T} |M^{n+1}(    \tau)|^{s-2}||DM^{n+1}(    \tau)||f_0(    \tau)|dx.
	\end{split}
\end{equation}
By Young inequality  with $1/s+1/s'=1$, we have
\begin{equation*} 
	\begin{split}
	&|\int_{\T}|M^{n+1}(  \tau)|^{s-2}DM^{n+1}(  \tau)f_0(  \tau)dx|\\
	\leq &\int_{\T}\big(|M^{n+1}(  \tau)|^{(s-2)/2} |DM^{n+1}(  \tau)|\big)^{s'}|M^{n+1}(  \tau)|^{s'(s-2)/2}dx\\
	+&C\|f_0(  \tau)\|^s_{L^s(\T)}.
	\end{split}
\end{equation*} 
By 
$$
\frac{s'(s-2)}{2}\cdot \frac{1}{1-s'/2}=s
$$
and 
\begin{align*}
	&\big(|M^{n+1}(  \tau)|^{(s-2)/2} |DM^{n+1}(  \tau)|\big)^{s'}|M^{n+1}(  \tau)|^{s'(s-2)/2}\\
	\leq &\frac{1}{2}\big( \big(|M^{n+1}(  \tau)|^{(s-2)/2} |DM^{n+1}(  \tau)|\big)^{s'}\big)^{2/s'}+C\big(|M^{n+1}(  \tau)|^{s'(s-2)/2})^{\frac{1}{1-s'/2}},
\end{align*}
we estimate
\begin{align*}
	& \int_{\T}\big(|M^{n+1}(  \tau)|^{(s-2)/2} |DM^{n+1}(  \tau)|\big)^{s'}|M^{n+1}(  \tau)|^{s'(s-2)/2}dx\\
	\leq & \frac{1}{2}\| |M^{n+1}(  \tau)|^{(s-2)/s}|DM^{n+1}(  \tau)|^{2/s}\|^s_{L^s(\T)}+C\|M^{n+1}(  \tau)\|^s_{L^s(\T)}.
\end{align*}
Moreover
\begin{align*}
	&\|f_0(  \tau)\|^s_{L^s(\T)}\\
	\leq &2^{s-1}\big(\|q^{(n+1)}(  \tau)M^{n+1}(  \tau)\|^s_{L^s(\T)}+\|(q^{(n+1)}(  \tau)-q^{(n)}(  \tau))m^{(n)}(  \tau)\|^s_{L^s(\T)}\big)\\
	\leq &C\big(\|M^{n+1}(  \tau)\|^s_{L^s(\T)}+\|q^{(n+1)}(  \tau)-q^{(n)}(  \tau)\|^s_{L^\infty(\T)}\big).
\end{align*}
Replacing the previous estimate in \eqref{id1}, we obtain 
\begin{equation*}
	\begin{split}
		&\pd_\tau \|M^{n+1}(  \tau)\|^s_{L^s(\T)}+ \frac{1}{2}s(s-1)\| |M^{n+1}(  \tau)|^{(s-2)/s}|DM^{n+1}(  \tau)|^{2/s}\|^s_{L^s(\T)}\\
		\leq &C\big(\|M^{n+1}(   \tau)\|^s_{L^s(\T)}+\|q^{(n+1)}(   \tau)-q^{(n)}(   \tau)\|^s_{L^\infty(\T)}\big).
	\end{split}
\end{equation*}
and, using  Gronwall's inequality, we finally  get
$$
\sup_{0<t<T}\|M^{n+1}(  t)\|_{L^s(\T)}\leq C\|q^{(n+1)}-q^{(n)}\|_{L^\infty(Q)}.
$$
We then proceed to show \eqref{mH12}. Multiplying both sides of \eqref{eq1} by $M^{n+1}$ and integrating in $\T$, we get
 \begin{equation}\label{eq2}
		\begin{split}
	&\half \frac{d}{dt}\|M^{n+1}(t)\|^2_{L^2(\T)}+\int_\T|DM^{n+1}(x,t)|^2dx\\
	=&-\int_\T q^{(n+1)}(x,t)M^{n+1}(x,t)DM^{n+1}(x,t)dx\\
	&-\int_\T(q^{(n+1)}(x,t)-q^{(n)}(x,t)) m^{(n)} (x,t)DM^{n+1}(x,t)dx.
 \end{split}
	\end{equation}
Next we integrate both sides of equation \eqref{eq2} on the interval $[0,T]$, and get
 \begin{equation*}
		\begin{split}
	&\|M^{n+1}(T)\|^2_{L^2(\T)}-\|M^{n+1}(0)\|^2_{L^2(\T)}+\int_0^T\int_\T|DM^{n+1}|^2dxdt\\
	=&-\int_0^T\int_\T q^{(n+1)}M^{n+1}DM^{n+1}dxdt\\
	&-\int_0^T\int_\T(q^{(n+1)}-q^{(n)}) m^{(n)} DM^{n+1}dxdt,
 \end{split}
	\end{equation*}
from which, recalling that $M^{n+1}( 0)\equiv 0$, we get
 \begin{equation}\label{eq:est2}
		\begin{split}
	&\int_0^T\int_\T|DM^{n+1}|^2dxdt 
	\leq \int_0^T\int_\T |q^{(n+1)}M^{n+1}DM^{n+1}|dxdt\\
	&+\int_0^T\int_\T|(q^{(n+1)}-q^{(n)}) m^{(n)} DM^{n+1}|dxdt.
 \end{split}
	\end{equation}	
Using \eqref{mCs} and Young inequality, we  estimate
  \begin{equation}\label{eq:est3}
		\begin{split}
	&\int_0^T\int_\T |q^{(n+1)}M^{n+1}DM^{n+1}|dxdt\\
	\leq &\frac{1}{4}\int_0^T\int_\T|DM^{n+1}|^2dxdt+\frac{3}{4}\int_0^T\int_\T |q^{(n+1)}M^{n+1}|^2dxdt\\
	\leq &\frac{1}{4}\int_0^T\int_\T|DM^{n+1}|^2dxdt+C\|q^{(n+1)}-q^{(n)}\|^2_{L^{\infty}(Q)},
 \end{split}
	\end{equation}	
and
 \begin{equation}\label{eq:est4}
		\begin{split}
	&\int_0^T\int_\T|(q^{(n+1)}-q^{(n)}) m^{(n)} DM^{n+1}|dxdt\\
\leq &\frac{1}{4}\int_0^T\int_\T|DM^{n+1}|^2dxdt+\frac{3}{4}\int_0^T\int_\T |(q^{(n+1)}-q^{(n)})m^{(n)}|^2dxdt\\
	\leq &\frac{1}{4}\int_0^T\int_\T|DM^{n+1}|^2dxdt+C\|q^{(n+1)}-q^{(n)}\|^2_{L^{\infty}(Q)}.
 \end{split}
	\end{equation}	
Replacing \eqref{eq:est3} and  \eqref{eq:est4} in \eqref{eq:est2}
$$
\frac{1}{2}\int_0^T\int_\T|DM^{n+1}|^2dxdt \leq C\|q^{(n+1)}-q^{(n)}\|^2_{L^{\infty}(Q)},
$$
so that, with \eqref{mCs}, we get
\begin{equation}\label{stima1}
\|m^{(n+1)} - m^{(n)}\|_{W^{1,0}_2(Q)}\le C \|q^{(n+1)}-q^{(n)}\|_{L^\infty (Q)}.
\end{equation}

For any test function $\phi\in W^{1,0}_{2}(Q)$, we  multiply both sides of \eqref{eq1} by $\phi$ and integrate on $Q$ to obtain
\begin{equation}\label{eq:est7}
\begin{split}
	&\int_0^T\int_{\T}\pd_t M^{n+1}  \phi dxdt=- \int_0^T\int_{\T}D M^{n+1}\cdot D \phi dxdt\\
	&-\int_0^T\int_{\T}D \phi \cdot (q^{(n+1)} M^{n+1})dxdt-\int_0^T\int_{\T}D \phi \cdot ((q^{(n+1)}-q^{(n)}) m^n)dxdt.
 \end{split}
\end{equation}
We estimate the three terms on the right hand side of \eqref{eq:est7} by
\begin{align*}
& \int_0^T\int_{\T}D M^{n+1}\cdot D \phi dxdt\\
 \leq &\Big(\int_0^T\int_\T|DM^{n+1}|^2dxdt\Big)^{\frac{1}{2}} \Big(\int_0^T\int_\T|D\phi|^{2}dxdt\Big)^{\frac{1}{2}}\\
 \leq &C\|q^{(n+1)}-q^{(n)}\|_{L^{\infty}(Q)}\|D\phi\|_{L^2(0,T;L^2(\T))},
\end{align*}
\begin{align*}
& \int_0^T\int_{\T}D \phi \cdot (q^{(n+1)} M^{n+1}) dxdt\\
 \leq &R\Big(\int_0^T\int_\T|M^{n+1}|^2dxdt\Big)^{\frac{1}{2}} \Big(\int_0^T\int_\T|D\phi|^2dxdt\Big)^{\frac{1}{2}}\\
 \leq &C\|q^{(n+1)}-q^{(n)}\|_{L^{\infty}(Q)}\|D\phi\|_{L^2(0,T;L^2(\T))},
\end{align*} 
and
\begin{align*}
& \int_0^T\int_{\T}D \phi \cdot ((q^{(n+1)}-q^{(n)}) m^n) dxdt\\
 \leq &\|q^{(n+1)}-q^{(n)}\|_{L^{\infty}(Q)} \left(\int_0^T\int_\T|m^{n}|^2dxdt\right)^{\frac{1}{2}} \left(\int_0^T\int_\T|D\phi|^2dxdt\right)^{\frac{1}{2}}\\
 \leq &C\|q^{(n+1)}-q^{(n)}\|_{L^{\infty}(Q)}\|D\phi\|_{L^2(0,T;L^2(\T))}.
\end{align*}
Replacing in \eqref{eq:est7}, we obtain 
\begin{equation*}
\sup_{\phi \in W^{1,0}_2(Q)}\int_0^T\int_{\T}\pd_t M^{n+1}  \phi dxdt\leq C\|q^{(n+1)}-q^{(n)}\|_{L^{\infty}(Q)}\|\phi\|_{W^{1,0}_2(Q)},
\end{equation*}
i.e.
\begin{equation}\label{stima2}
	\|\pd_t M^{n+1}\|_{(W^{1,0}_2(Q))'}\leq C\|q^{(n+1)}-q^{(n)}\|_{L^\infty (Q)}.
\end{equation}
From  \eqref{stima1} and \eqref{stima2}, we get \eqref{mH12}.\\
\indent We now prove the estimate \eqref{uW21r} for the HJB equation. The function  $U^{n+1}=u^{(n+1)}-u^{(n)}$ satisfies the equation
\begin{equation*}
-\pd_t U^{n+1}- \Delta U^{n+1}+q^{(n+1)} DU^{n+1}=\cF(x,t)
\end{equation*}
with $U^{n+1}(x,T)=0$, where 
\begin{equation}\label{cF}
\cF(x,t)=\sigma F[m^{(n+1)}]-\sigma F[m^{(n)}]+q^{(n)}Du^{(n)}-L(q^{(n)})
-(q^{(n+1)}Du^{(n)}-L(q^{(n+1)})). 
\end{equation}
Hence, recalling that $ q^{(n+1)} =H_p(Du^{(n)})$ is bounded, we have the estimate (see Lemma \ref{lemma:linear_HJ})
\begin{equation}\label{eq:conv1}
\|U^{n+1}\|_{W^{2,1}_r(Q)}\le  \|\cF \|_{L^r(Q)}.
\end{equation}
To estimate $\|\cF\|_{L^r(Q)}$, first observe that, since
\begin{align*}
&q^{(n+1)}Du^{(n)}-L(q^{(n+1)})= \sup_{q\in\R^d}\left\{q\cdot Du^{(n)}-L(q)\right\}\\
&\ge q^{(n)}Du^{(n)}-L(q^{(n)}),
\end{align*}
then we have
\begin{equation}\label{eq:conv2}
\cF(x,t)\le \sigma( F[m^{(n+1)}](x)- F[m^{(n)}](x)).
\end{equation}
Moreover, by $q^{(n)}= H_p(Du^{(n-1)})$ and 
\[H(Du^{(n-1)})=q^{(n)}Du^{(n-1)}-L(q^{(n)}),\]
 we have
\begin{align*}
&q^{(n)}Du^{(n)}-L(q^{(n)})-(q^{(n+1)}Du^{(n)}-L(q^{(n+1)}))=\\
&q^{(n)}Du^{(n)}+H(Du^{(n-1)})-q^{(n)}Du^{(n-1)} -H(Du^{(n)})=\\
& H(Du^{(n-1)})- H(Du^{(n)})+H_p(Du^{(n-1)})(Du^{(n)}-Du^{(n-1)})=\\
& -\half H_{pp}(\th Du^{(n)}+(1-\th)Du^{(n-1)} )(Du^{(n)}-Du^{(n-1)})\cdot( Du^{(n)}-Du^{(n-1)})
\end{align*}
for some $\th\in (0,1)$. Therefore,  either by \textbf{(H3)} or by \textbf{(H2)}, recalling   that  $\|Du^{(n)}\|_\infty$ is  uniformly bounded in $n\in\N$, we get
\begin{equation}\label{eq:conv3}
\cF(x,t)\ge \sigma (F[m^{(n+1)}]- F[m^{(n)}])-C|Du^{(n)}(x,t)-Du^{(n-1)}(x,t)|^2.
\end{equation}
%
Recall that $r>d+2$. From  \eqref{eq:conv2}, \eqref{eq:conv3} and   Minkowski inequality we obtain
\begin{align*}
 \|\cF \|_{L^r(Q)}\le & C\|Du^{(n)}-Du^{(n-1)}\|^2_{L^{2r}(Q)}+\sigma \|F[m^{(n+1)}]-F[m^{(n)}]\|_{L^r(Q)}\\
\le & C \big( \|u^{(n)}-u^{(n-1)}\|^2_{W^{2,1}_r(Q)}+\sigma \|F[m^{(n+1)}]-F[m^{(n)}]\|_{L^r(Q)}\big)
\end{align*}
From \textbf{(F2)}, for each $t$  we have  
$$
\|F[m^{(n+1)}](t)-F[m^{(n)}](t)\|_{L^r(\T)}\le C \|m^{(n+1)}( t)-m^{(n)}( t)\|_{L^s(\T)},
$$
so that 
\begin{align*}
\|F[m^{(n+1)}]-F[m^{(n)}]\|_{L^r(Q)}\le & C\|F[m^{(n+1)}]-F[m^{(n)}]\|_{L^\infty(0,T;L^r(\T))}\\
\le &  C \|m^{(n+1)}-m^{(n)}\|_{C(0,T;L^s(\T))}.
\end{align*}
Then we can get \eqref{uW21r} from \eqref{eq:conv1}.

\end{proof}


 A key difficulty for estimating convergence rate using Theorem \ref{prop:est1} is that we cannot control the constants $C$ in \eqref{mH12} and \eqref{uW21r}. We do not have additional information other than they depend on the data of the problem and not on $n$. To address this difficulty, we introduce an additional assumption on the smallness of $\sigma$. It is not needed for  the convergence of the policy iteration method but allows us to get a linear convergence rate to the solution of MFG system in the policy iteration. This type of assumption also plays a key role in \cite{cg} for considering MFGs of aggregation.

\begin{corollary}\label{cor:rate}
Under the same assumptions of Theorem \ref{prop:est1}, there exist constants $\sigma_0>0$ and $0<C^*<1$, such that for sufficiently large $n$ and $\forall \sigma<\sigma_0$, 
 \begin{equation}\label{estimate_rate}
		\|u^{(n+1)}-u^{(n)}\|_{W^{2,1}_r(Q)}+ \s \|m^{(n+1)} - m^{(n)}\|_{C(0,T;L^s(\T))} \leq C^*\|u^{(n)}-u^{(n-1)}\|_{W^{2,1}_r(Q)}.
		\end{equation}
		
\end{corollary} 
\begin{proof}
	First note that, by parabolic Sobolev embedding theorem (\cite{LSU}, Corollary IV.9.1 p.342) and the fact that $r>d+2$, we have 
\begin{equation*}
\|q^{(n+1)}-q^{(n)}\|_{L^\infty(Q)}\leq C\|u^{(n)}-u^{(n-1)}\|_{W^{2,1}_r(Q)}.
\end{equation*}	
	By \eqref{uW21r} and \eqref{mH12}  we have
\begin{align*}
& \|m^{(n+1)} - m^{(n)}\|_{C(0,T;L^s(\T))}  
\le C\|q^{(n+1)}-q^{(n)}\|_{L^{\infty}(Q)}\\
&\le C_1\|u^{(n)}-u^{(n-1)}\|_{W^{2,1}_r(Q)}.
\end{align*}
From \eqref{uW21r} we have
\begin{equation*}
		\begin{split}
	 &\|u^{(n+1)}-u^{(n)}\|_{W^{2,1}_r(Q)}
	 \le C_2\big(\|u^{(n)}-u^{(n-1)}\|^2_{W^{2,1}_r(Q)}\\
	 &+ \sigma \|m^{(n+1)} - m^{(n)}\|_{C(0,T;L^s(\T))} \big).
	 \end{split}
	\end{equation*}
Here $C_1$ and $C_2$ are always independent of $n$ and $\sigma$.   Then for sufficiently small $\sigma$ we have $C_2C_1\sigma<\frac{C^*}{4}$ and 
\begin{equation*}
\begin{split}
&\|u^{(n+1)}-u^{(n)}\|_{W^{2,1}_r(Q)}+\s \|m^{(n+1)} - m^{(n)}\|_{C(0,T;L^s(\T))}\\
\leq &C_2\|u^{(n)}-u^{(n-1)}\|^2_{W^{2,1}_r(Q)}+2C_2C_1\sigma \|u^{(n)}-u^{(n-1)}\|_{W^{2,1}_r(Q)}.
\end{split}
\end{equation*}
Since $u^{(n)}$ converges in $W^{2,1}_r(Q)$ and $C_2$ is independent of $n$, we have for sufficiently large $n$, 
$$
C_2\|u^{(n)}-u^{(n-1)}\|_{W^{2,1}_r(Q)}< \frac{C^*}{2},
$$
so that

\begin{align*}
\|u^{(n+1)}&-u^{(n)}\|_{W^{2,1}_r(Q)}+ \s \|m^{(n+1)} - m^{(n)}\|_{C(0,T;L^s(\T))}\\
&	\leq C^*\|u^{(n)}-u^{(n-1)}\|_{W^{2,1}_r(Q)}
\end{align*}
\end{proof}

\begin{remark}
	Assumption \textbf{(F2)} is satisfied for example if
	\[
	F[m] =\int_\T K(x,y)m(y,t)dy,
	\]
	for some bounded kernel $K:\T\times\T\to\R$. We have
\begin{align*}	
	 \|F[m_1(t)]-F[m_2(t)]\|_{L^r(\T)}&\le C \|F[m_1(t)]-F[m_2(t)]\|_{L^\infty(\T)}\\
	 &\le C  \|K\|_{L^\infty(\T \times \T)}\|m_1( t)-m_2( t)\|_{L^s(\T)}.	
\end{align*}
If $K(x,y)=K(x-y)$, it is sufficient to assume that $K\in L^\zeta(\T)$, with $1/s+1/\zeta=1/r+1$. Indeed, by Young's convolution inequality, we have in this case
\begin{align*}
\|F[m_1(t)]-F[m_2(t)]_{C(0,T;L^s(\T))}]\|_{L^r(\T)}\\
\le \|K\|_{L^\zeta(\T\times\T )}\|m_1( t)-m_2( t)\|_{L^s(\T)}.
\end{align*}
Note that  estimate \eqref{estimate_rate} also holds for the case of a local coupling, i.e. $F=F(x,m)$, assuming  $F$ to be Lipschitz continuous in $m$, uniformly in $x$ and $r\leq s$.
\end{remark}


\begin{remark}
The results can be generalized to a  Hamiltonian $H$ dependent on $x\in\T$,
where the assumptions (H1), (H2), (H3) are replaced respectively  by 
\begin{itemize}
	\item[\bf{($\widetilde{H1}$)}] $H$ is differentiable, convex and globally Lipschitz  continuous, i.e.  there exists  a constant $R_0>0$ such that
	\begin{equation*} 
		|D_pH(x,p)|\leq R_0\qquad\text{ for all }p\in\R^d\ ,
	\end{equation*}
	uniformly in $x$.
	\item[\textbf{($\widetilde{H2}$)}] $H$ is of the form
	\begin{equation*} 
		H(x,p)=h(x)|p|^\gamma, \qquad \gamma>1,
	\end{equation*}
	where $0<h_0<h(x)<h_1$, $h_0$ and $h_1$ are two constants. 
	\item[\textbf{($\widetilde{H3}$)}]  $H$  is two times differentiable, satisfies $(\widetilde{H1})$ and for any $S>0$, there  exists $C_S>0$ such that
	\begin{equation*} 
		H_{pp}(x,p)q\cdot q\le C_S|q|^2 \quad \text{for any    $|p|\le S$, $q\in\R^d$},
	\end{equation*}
	uniformly in $x$.
\end{itemize}
When either $(\widetilde{H1})$ or $(\widetilde{H2})$ holds, the uniform boundedness of $Du$ has been shown in \cite{cG}. Moreover, either by $(\widetilde{H2})$ or by $(\widetilde{H3})$, we get \eqref{eq:conv3}.
\end{remark}

\begin{remark}
We can also generalize the results to include the case 
$$
u(x,T)=\sigma'u_T[m],
$$
where $\sigma'$ is a positive constant,  i.e. the final cost depends on  the agents state distribution, assuming the regularizing assumption 
\begin{equation}\label{reg_assumption}	
\|u_T[m_1]-u_T[m_2]\|_{W^{2-\frac{2}{r},r}(\T)} \leq C\|m_1(T)-m_2(T)\|_{L^s(\T)}.
\end{equation}
The function  $U^{n+1}=u^{(n+1)}-u^{(n)}$ satisfies the equation	
\begin{equation*}
-\pd_t U^{n+1}- \Delta U^{n+1}+q^{(n+1)} DU^{n+1}=\cF(x,t),
\end{equation*}
with $U^{n+1}(x,T)=\sigma'(u_T[m^{(n+1)}]-\sigma'u_T[m^{(n)}])$, $\cF$ defined as \eqref{cF}. Then, using Lemma \ref{lemma:linear_HJ}  and \eqref{reg_assumption} to estimate 
\begin{align*}
\|u_T[m^{(n+1)}]-u_T[m^{(n)}]\|_{W^{2-\frac{2}{r},r}(\T)} \leq &C\sigma'\|m^{(n+1)}(T)-m^{(n)}(T)\|_{L^s(\T)}\\
\leq &C\sigma'\|m^{(n+1)}-m^{(n)}\|_{C(0,T;L^s(\T))}
\end{align*}
  we will have 
\begin{equation} 
		\begin{split}
	 \|u^{(n+1)}-u^{(n)}\|_{W^{2,1}_r(Q)}
	 &\le C\big(\|u^{(n)}-u^{(n-1)}\|^2_{W^{2,1}_r(Q)}\\
	 &+ (\sigma+\sigma') \|m^{(n+1)} - m^{(n)}\|_{C(0,T;L^s(\T))}\big).
	 \end{split}
	\end{equation}
Hence, we get a linear rate of convergence    if we assume both $\sigma$ and $\sigma'$ sufficiently small.
\end{remark}

\begin{remark}
	Assume that $F$ is independent of $m$, i.e. $F[m](x)=F(x)$. In this case, by Proposition \ref{prop:est1},  we recover two well known properties  of  the policy iteration method for the Hamilton-Jacobi-Bellman equation. Firstly, by \eqref{eq:conv2}, we have
	that $U^{n+1}=u^{(n+1)}-u^{(n)}$  
	satisfies
	\[
	-\pd_t U^{n+1}- \Delta U^{n+1}+q^{(n+1)} DU^{n+1}\le 0
	\]
	with  $U^{n+1}(x,T)\equiv 0$. Therefore, by comparison principle, $U^{n+1}\le 0$, hence the policy iteration method  generates a decreasing sequence $u^{(n)}$. Moreover, by estimate \eqref{uW21r}, we get a (local) quadratic rate of convergence  for the method (a similar estimate is proved in \cite{kss} via probabilistic techniques). 
\end{remark}
\section{A rate of convergence for the policy iteration method:  the ergodic problem}\label{sec:stat}
In this section, we prove a rate of convergence for the policy iteration method for the  the ergodic MFG system
\begin{equation}\label{MFG_stat}
	\begin{cases}
		- \Delta u+H(Du)+\l=\sigma F[m] & \text{ in }\T\\
		- \Delta  m-\diver(m H_p(Du))=0 & \text{ in }\T\\
		\int_\T m(x)dx=1,\quad m\ge 0,\quad \int_\T u(x)dx=0 \ .
	\end{cases}
\end{equation}
\textbf{Policy iteration algorithm:} For fixed $R>0$ and given a bounded, measurable function $q^{(0)}$ such that $\|q^{(0)}\|_{L^\infty(\T)}\le R$,
a policy iteration method for \eqref{MFG_stat} is given by
\begin{itemize}
	\item[(i)]  Solve
	\begin{equation*}
		\left\{
		\begin{array}{ll}
			- \Delta m^{(n)}-\diver (m^{(n)} q^{(n)})=0,\quad &\text{ in }\T\\
			\int_\T m^{(n)}(x)dx=1,\quad m^{(n)}\ge 0.
		\end{array}
		\right.
	\end{equation*}
	\item[(ii)] Solve
	\begin{equation*}
		\left\{
		\begin{array}{ll}
			- \Delta u^{(n)}+q^{(n)}\cdot  Du^{(n)}-L(q^{(n)})+\l^{(n)}=\sigma F[m^{(n)}]&\text{ in }\T\\
			\int_\T u^{(n)}(x)dx=0.
		\end{array}
		\right.
	\end{equation*}
	\item[(iii)] Update the policy
	\begin{equation*}
		q^{(n+1)}(x,t)={\arg\max}_{|q|\le R}\left\{q\cdot Du^{(n)}(x)-L(q)\right\}\quad\text{ in }\T.
	\end{equation*}
\end{itemize}
In \cite[Section 4]{ccg}, the following convergence theorem   is proved
\begin{theorem}\label{thm:convergence_ergodic}
	Let either \textbf{(H1)} or \textbf{(H2)}  and  \textbf{(F1)}  be  in force and $R$ sufficiently large. Then, the sequence $(u^{(n)},\l^{(n)}, m^{(n)})$, generated by the policy iteration algorithm converges    to the solution  $(u,\l ,m)\in W^{2,r}(\T)\times\R\times   W^{1,s}(\T)$  of \eqref{MFG_stat}, uniformly in $\T$.
\end{theorem} 
For the proof of the convergence estimate, we need a preliminary lemma. 
\begin{lemma}\label{stat lemma}
	Let $f\in W^{-1,s}(\T)=(W^{1,s'}(\T))'$, $s'=s/(s-1)$, and $q\in L^\infty(\T)$. If $M$ satisfies 
\begin{equation}\label{stationary FP2}
	\left\{
	\begin{array}{ll}
		-\Delta M-\diver (qM)=f,\quad &\text{ in } \T\\
		\int_{\T}Mdx=0,
	\end{array}
	\right.
\end{equation}	
then
	\begin{equation}\label{M estimate}
		\|M\|_{W^{1,s}(\T)} \leq C \|f\|_{W^{-1,s}(\T)}.
	\end{equation}
\end{lemma}
\begin{proof}
From   \cite[Prop. 1.2.3]{bkrs})                                                                                                                                                
\begin{equation}
\begin{split}
	\|M\|_{W^{1,s}(\T)} \leq &C\|\diver(qM)+f\|_{W^{-1,s}(\T)}\\
	\leq &C\big(\|q\|_{L^\infty(\T)}\|M\|_{ L^s(\T)}+\|f\|_{W^{-1,s}(\T)}\big)\\
	\leq &C( \|M\|_{ L^s(\T)}+\|f\|_{W^{-1,s}(\T)}).
\end{split}
\end{equation}
We claim this leads to \eqref{M estimate}. Following the argument in \cite[pag.6]{SchechterI}, we assume by contradiction that there exists a sequence $M_k$ of solutions to \eqref{stationary FP2} such that 
	$$
	\|M_k\|_{W^{1,s}(\T)}=1,\,\,\|AM_k\|_{W^{-1,s}(\T)}\rightarrow 0,\,\,\int_{\T}M_kdx=0.
	$$
where
$$
A:=-\Delta \cdot -\diver (q\cdot).
$$
	By Rellich-Kondrachov Theorem, $W^{1,s}(\T)$ is compactly embedded in $L^s(\T)$ for $1<s<\infty $. Then from Banach-Alaoglu theorem there is a subsequence, again denoted by $M_k$, which converges weakly in $W^{1,s}(\T)$ and strongly in $L^s(\T)$. We have
	\begin{equation*}
		\|M_j-M_k\|_{W^{1,s}(\T)}\leq C(\|AM_j-AM_k\|_{W^{-1,s}(\T)}+\|M_j-M_k\|_{ L^s(\T)}).
	\end{equation*}
	Then we have that $M_k$ converges to $\bar M$ in $W^{1,s}(\T)$ and
	$$
	A\bar M=0,\,\,\,\int_{\T}\bar Mdx=0,\,\,
	$$
 By Theorem 4.2 and Lemma 4.3 in \cite{bensoussan},  the previous problem has only a trivial solution
 and therefore a contradiction since $\|\bar M\|_{W^{1,s}(\T)}=1$.
\end{proof}

\begin{theorem}\label{prop:est_stat}
	Let either  \textbf{(H2)} or \textbf{(H3)}, \textbf{(F2)}  be  in force, $r,s>d$ and $R$  as in Theorem \ref{thm:convergence_ergodic}. Then, there exists a constant $C$, depending only on the data of problem, such that, if $(u^{(n)}, \l^{(n)}, m^{(n)})$ is the sequence generated by the policy iteration method, we have
	\begin{equation}\label{st12}
			\|m^{(n+1)} - m^{(n)}\|_{W^{1,s}(\T)}\le C \|q^{(n+1)}-q^{(n)}\|_{ L^\infty (\T)},
	\end{equation}

	\begin{equation} \label{st13}
		\begin{split}
			\|u^{(n+1)}-u^{(n)}\|_{W^{2,r} (\T)}+|\l^{(n+1)}-\l^{(n)}|
			\le C\big(\|u^{(n)}-u^{(n-1)}\|^2_{W^{2,r} (\T)}\\
			+ \sigma \|m^{(n+1)} - m^{(n)}\|_{W^{1,s}(\T)}\big),
		\end{split}
	\end{equation}
\end{theorem}
\begin{proof}
Along the proof, the constant $C$ can change from line to line, but it is always independent of $n$.\\
Set $M^{n+1}=m^{(n+1)}-m^{(n)}$. Then $M^{n+1}$ satisfies the equation
\begin{equation*}
	- \Delta M^{n+1}-\diver (q^{(n+1)} M^{n+1})=\diver ((q^{(n+1)}-q^{(n)}) m^{(n)}).
\end{equation*}
with $\int_{\T}M^{n+1}dx=0$. Hence, by \eqref{M estimate} we have
\begin{align*}
	\|M^{n+1}\|_{W^{1,s}(\T)}\leq &C\|\diver ((q^{(n+1)}-q^{(n)}) m^{(n)})\|_{W^{-1,s}(\T)}\\
	\leq &C\|(q^{(n+1)}-q^{(n)}) m^{(n)}\|_{L^s(\T)}\\
	\leq &C\|(q^{(n+1)}-q^{(n)})\|_{L^{\infty}(\T)} \|m^{(n)}\|_{L^s(\T)} \\
	\leq & C\|(q^{(n+1)}-q^{(n)})\|_{L^{\infty}(\T)}
\end{align*}
and  therefore  \eqref{st12}.\\
The couple   $U^{n+1}=u^{(n+1)}-u^{(n)}$, $\L^{n+1}=\l^{(n+1)}-\l^{(n)}$ satisfies the equation
\begin{equation*}
	- \Delta U^{n+1}+q^{(n+1)} DU^{n+1}+\L^{n+1}=\cF(x)
\end{equation*}
with $\int U^{n+1}(x)dx=0$, where 
\[ 
\cF(x)=\sigma (F[m^{(n+1)}]-  F[m^{(n)}])+q^{(n)}Du^{(n)}-L(q^{(n)})
-(q^{(n+1)}Du^{(n)}-L(q^{(n+1)})). 
\]
Exploiting the results in \cite{bensoussan} (Theorem 6.1, Pag. 196), we get
\begin{align}
	|\L^{n+1}|\le \|\cF(x)\|_{L^r(\T)}\label{st15}\\
	\|U^{n+1}\|_{W^{2,r}(\T)}	\le C\|\cF(x)\|_{L^r(\T)}\label{st16}
\end{align}	
Repeating similar computations to the one of the corresponding estimate in Theorem \ref{prop:est1}, we obtain
\begin{align*}
	\|\cF(x)\|_{L^r(\T)}\le  C \big(  \|Du^{(n)}-Du^{(n-1)}\|^2_{L^\infty(\T)}+\sigma \|F[m^{(n+1)}]-F[m^{(n)}]\|_{L^r(\T)}\big).
\end{align*}
Replacing the previous estimate in \eqref{st15}-\eqref{st16} and exploiting  \textbf{(F2)} and $r,s>d$, we get
\eqref{st13}.
\end{proof}
\begin{remark} 
	For $s=2$, estimate   \eqref{st12}  is a special case of Lemma 3.8 from \cite{all}.
\end{remark}
Arguing as in Corollary \ref{cor:rate}, we can   obtain the rate of convergence 

\begin{corollary}
	Under the same assumptions of Theorem \ref{prop:est_stat}, there exist constants $\sigma_1$ and $0<C^{**}<1$, such that for $\sigma<\sigma_1$ 
	 
	\begin{equation*}
	\begin{split}
		&\|u^{(n+1)}-u^{(n)}\|_{W^{2,r} (\T)}+|\l^{(n+1)}-\l^{(n)}|+\sigma \|m^{(n+1)} - m^{(n)}\|_{W^{1,s}(\T)}\\
		\leq &C^{**} \|u^{(n)}-u^{(n-1)}\|^2_{W^{2,r} (\T)}.
		\end{split}
			\end{equation*}
\end{corollary} 


\section{An interpretation of the policy iteration method for the MFG system} \label{sec:Newton_MFG}
	 The following computations only hold at a formal level and, for simplicity, we assume that   $F$ is local coupling.  Define the map
	 $$\cF:(u,m)\to
	 \left( 
	 \begin{array}{l}
	 	-\partial_tu-\Delta u+H(Du)-\sigma F(m)\\
	 	\partial_tm -\Delta  m-\diver(mH_p(Du))\\
	 	u(T)-u_T(x)\\
	 	m(0)-m_0(x)
	 \end{array}
 \right).
 $$	
 Then system \eqref{MFG} is equivalent to find the roots of $\cF$ and the corresponding Newton's iterations
 can be written as
 \begin{equation}\label{eq:Newton_iter}
 	J\cF(u^{(n-1)},m^{(n-1)})((u^{(n)},m^{(n)})-(u^{(n-1)},m^{(n-1)}))=- \cF(u^{(n-1)},m^{(n-1)}).
 \end{equation}
 The Jacobian  of $\cF$ is given by
 \[
 J\cF(u,m)(\cdot,\cdot)=\left( 
 \begin{array}{lr}
 	-\partial_t\cdot -\Delta \cdot+H_p(Du)\cdot\quad &-\sigma F'(m)\cdot\\[4pt]
     -\diver (mH_{pp}D\cdot)                      	&\partial_t\cdot -\Delta  \cdot-\diver(H_p(Du)\cdot)\\
 	\cdot|_{t=T}-u_T(x)&0\\
 	0&\cdot|_{t=0}-m_0(x)
 \end{array}
 \right).
 \]
where $F'=\frac{dF}{dm}$. 
Replacing in \eqref{eq:Newton_iter}, we obtain for the first component of \eqref{eq:Newton_iter}
 \begin{align*}
 	&-\pd_t(u^{(n)}-u^{(n-1)})-\Delta(u^{(n)}-u^{(n-1)})+H_p(Du^{(n-1)})(u^{(n)}-u^{(n-1)})\\
 	&-\sigma F'(m^{(n-1)})(m^{(n)}-m^{(n-1)})\\
	&=\partial_tu^{(n-1)}+\Delta u^{(n-1)}-H(Du^{(n-1)})+\sigma F(m^{(n-1)})
 \end{align*}
Now recalling that $q^{(n)}=H_p(Du^{(n-1)})$ and $H(Du^{(n-1)})=q^{(n)}Du^{(n-1)}-L(q^{(n)})$, the previous equation is equivalent to
\begin{equation}\label{eq:Newton_1cond}
	\begin{split}
&-\pd_t u^{(n)}- \Delta u^{(n)}+q^{(n)} Du^{(n)}-L(q^{(n)})=\sigma F(m^{(n-1)})\\
&+\sigma F'(m^{(n-1)})(m^{(n)}-m^{(n-1)})
	\end{split}	
\end{equation}
with the final condition $u^{(n)}(x,T)=u_T(x)$. Note that, if $F=F(x)$ and therefore $F'\equiv 0$, we see by 
\eqref{eq:Newton_1cond}  that the policy iteration method is a Newton's method applied to the HJB equation.
 With similar computation we get 
\begin{equation}\label{eq:Newton_2cond}
	\begin{split}
		&\pd_t m^{(n)}-\Delta m^{(n)}-\diver (m^{(n)} q^{(n)})\\
		&=\diver (m^{(n-1)}H_{pp}(Du^{(n-1)})(m^{(n)}-m^{(n-1)}))
	\end{split}	
\end{equation}
with the initial condition $m^{(n)}(x,0)=m_0(x)$. \par
By Theorem \ref{thm:policy_iteration}, the terms on the right side of \eqref{eq:Newton_1cond}-\eqref{eq:Newton_2cond}, which correspond to the off-diagonal entries of the Jacobian $J\cF$,  are infinitesimal. In the policy iteration method, we  suppress   these terms from the beginning, in order to remove the coupling between the two equations. In this sense, the policy iteration method can be interpreted as a quasi-Newton method since, instead of the full Jacobian of $\cF$, we only use an approximation of it. In any case, after some  iterations, the influence of the neglected 
terms  is vanishing and the  policy iteration method approximately behaves as   a Newton's method,
explaining the rapid convergence observed experimentally (see \cite[Section 6]{ccg}).\par
Some numerical examples have been considered in \cite{lst} for comparing the policy iteration method and the Newton method for solving MFGs. In many of these the policy iteration method turns out to be more efficient in terms of computing time.


\begin{flushright}
	\noindent \verb"fabio.camilli@uniroma1.it"\\
	SBAI, Sapienza Universit\`{a} di Roma\\
	  Roma (Italy)	
\end{flushright}

\begin{flushright}
	\noindent \verb"tangqingthomas@gmail.com"\\
	China University of Geosciences (Wuhan)\\
	Wuhan (China)	
\end{flushright}


\begin{thebibliography}{99}
	
\bibitem{all}
Achdou, Y.; Lauri{\`e}re, M.; Lions, P.L.
\newblock Optimal control of conditioned processes with feedback controls. J. Math. Pures Appl. (9) 148 (2021), 308-341.

\bibitem{alla}
Alla, A.; Falcone, M.; Kalise, D. An efficient policy iteration algorithm for dynamic programming equations. SIAM J. Sci. Comput. 37 (2015), no. 1, A181-A200. 
%
%
%
%
    
	\bibitem{b}
	Bellman, R. {\it Dynamic Programming}. Princeton Univ. Press, Princeton, 1957.
	
	\bibitem{bensoussan}
	Bensoussan, A.
	\newblock {\it Perturbation methods in optimal control}, Wiley/Gauthier-Villars Series in Modern Applied Mathematics. John Wiley \& Sons, Ltd., Chichester; Gauthier-Villars, Montrouge, 1988.
	
      \bibitem{ben}
    Bensoussan, A.; Frehse, J.; Yam, P.
 {\it  Mean field games and mean field type control theory}, Springer Briefs in Mathematics. Springer, New York, 2013. 

      \bibitem{bkrs}
Bogachev, V.~I; Krylov, N.~V.; R{\"o}ckner, M.;  Shaposhnikov, S.~V.
 {\it Fokker-Planck-Kolmogorov Equations}. Mathematical Surveys and Monographs, 207. American Mathematical Society, Providence, RI, 2015.

\bibitem{bmz}
	Bokanowski, O.; Maroso, S.; Zidani, H. Some convergence results for Howard's algorithm. SIAM J. Numer. Anal. 47 (2009), no. 4, 3001-3026. 
	
	\bibitem{bks}
	Brice\~{n}o-Arias, L. M.; Kalise, D.; Silva, F. J. Proximal methods for stationary mean
	field games with local couplings. SIAM J. Control Optim. 56 (2018), no. 2, 801-836.
	
	\bibitem{ccg}
	Cacace, S.;  Camilli, F.; Goffi, A. A policy iteration method for Mean Field Games, ESAIM Control Optim. Calc. Var.,     27 (2021), paper No. 85, 19 pp.
	
	 \bibitem{ch}
	Cardaliaguet, P.; Hadikhanloo, S. Learning in mean field games: the fictitious play. ESAIM Control Optim. Calc. Var. 23 (2017), no. 2, 569-591.


 \bibitem{c1999}
 Krylov, N. V. An analytic approach to SPDEs. Stochastic partial differential equations: six perspectives, 185-242, Math. Surveys Monogr., 64, Amer. Math. Soc., Providence, RI, 1999.



 \bibitem{cg}
Cirant, M.; Ghilli, D.  Existence and non-existence for time-dependent mean field games with strong aggregation. Math. Ann. (2021).
%
\bibitem{cG}
 Cirant, M.;  Goffi, A.  Lipschitz regularity for viscous Hamilton-Jacobi equations with Lp terms. Ann. Inst. H. Poincar\'e Anal. Non Lin\'eaire 37 (2020), no. 4, 757-784.



	\bibitem{fl}
	Fleming, W.~H. Some Markovian optimization problems. J. Math. Mech. 12 (1963), 131-140.	
		
	\bibitem{h}
	Howard, R. {\it Dynamic Programming and Markov Processes}. MIT Press, Cambridge, 1960.
	 
	\bibitem{hcm}
	Huang,  M.; Caines,  P.~E.; Malhame,  R.~P.  Large-population cost-coupled LQG problems with non uniform agents: Individual-mass behaviour and decentralized $\epsilon$-Nash equilibria. IEEE Transactions on Automatic  Control, 52 (2007), 1560-1571.
	
	\bibitem{kss}
	Kerimkulov, B.; \v{S}i\v{s}ka, D.;  Szpruch, L.
	Exponential convergence and stability of Howards's policy improvement algorithm for controlled diffusions, SIAM J. Control Optim. 53 (2020), 1314--1340.
	
	\bibitem{LSU}
	 Ladyzenskaja, O.~A.; Solonnikov,  V.~A.; Ural'ceva, N.~N.
     {\it Linear and quasilinear equations of parabolic type}.
	 Translated from the Russian by S. Smith. Translations of Mathematical
	 Monographs, Vol. 23. American Mathematical Society, Providence, R.I., 1968.
	 
	 \bibitem{lst}
Lauri{\`e}re, M.; Song, J.; Tang, Q.
\newblock Policy iteration method for time-dependent mean field games systems
  with non-separable Hamiltonians.
\newblock {arXiv:2110.02552}, 2021.


	 
	 \bibitem{metafune_altri}
	 Metafune, G.; Pallara, D.;  Rhandi A. Global properties of transition probabilities of singular diffusions.
	 Teor. Veroyatn. Primen., 54 (2009), 116--148.
	 
	 \bibitem{ll}
	 Lasry,  J.-M.;  Lions, P.-L. Mean field games. Jpn. J. Math. 2(2007), 229--260.
	
%
%

	\bibitem{pu1}
	Puterman, M.~L. On the convergence of policy iteration for controlled diffusions. J. Optim. Theory Appl. 33 (1981), no. 1, 137-144.
	
	
\bibitem{pb}
	Puterman, M.~L.; Brumelle, S.~L.
	On the convergence of policy iteration in stationary dynamic programming. 
	Math. Oper. Res. 4 (1979), 60-69.

\bibitem{santos}
Santos, M. S.; Rust, J. Convergence properties of policy iteration. 
SIAM J. Control Optim. 42 (2004), no. 6, 2094-2115.	

\bibitem{SchechterI}
Schechter, M.
\newblock On $ L^p$ estimates and regularity, I.
Amer. J. Math. 85 (1963), 1-13. 	

\end{thebibliography}
\end{document}